\documentclass{amsart}
\usepackage[T1]{fontenc}
\usepackage{amsmath}
\usepackage{amssymb}
\usepackage{tikz-cd}

\usepackage{xcolor}

\usepackage[style=alphabetic,backend=biber]{biblatex}
\newtheorem{theorem}[subsection]{Theorem}
\newtheorem{proposition}[subsection]{Proposition}
\newtheorem{definition}[subsection]{Definition}
\newtheorem{lemma}[subsection]{Lemma}

\newcommand{\chara}{{\textup{char}}}

\newcommand{\N}{{\mathbb N}}
\newcommand{\Z}{{\mathbb Z}}
\renewcommand{\P}{{\mathbb P}}
\newcommand{\Q}{{\mathbb Q}}

\title{All solid rings}
\author{Jaime Benabent Guerrero}
\date{\today}

\makeatletter
\def\@tocline#1#2#3#4#5#6#7{\relax
  \ifnum #1>\c@tocdepth 
  \else
    \par \addpenalty\@secpenalty\addvspace{#2}%
    \begingroup \hyphenpenalty\@M
    \@ifempty{#4}{%
      \@tempdima\csname r@tocindent\number#1\endcsname\relax
    }{%
      \@tempdima#4\relax
    }%
    \parindent\z@ \leftskip#3\relax \advance\leftskip\@tempdima\relax
    \rightskip\@pnumwidth plus4em \parfillskip-\@pnumwidth
    #5\leavevmode\hskip-\@tempdima
      \ifcase #1
       \or\or \hskip 1em \or \hskip 2em \else \hskip 3em \fi%
      #6\nobreak\relax
    \hfill\hbox to\@pnumwidth{\@tocpagenum{#7}}\par
    \nobreak
    \endgroup
  \fi}
\makeatother

\begin{document}
	
	\begin{abstract}
		We give an explicit characterization of all the solid rings, refining in this way a previous work by Bousfield and Kan in the 70s.
        As a consequence of our methods, we give criteria for finding the core of a ring and explicitly compute the core of some rings.
	\end{abstract}
	
	\maketitle

	
	\section{Introduction}
    \noindent


    In their seminal monograph, \cite{BK72}, Bousfield and Kan introduced the $R$-completion functor.
    This tool is characterized by the fact that a map between two spaces is a homology equivalence with coefficients in a ring $R$ if and only if the induced map between their respective $R$-completions is a homotopy equivalence.

    Their main interest at the time was to develop a way of isolating and studying the $p$-primary and rational information within the structure of a space.
    This is achieved by using coefficient rings such as $\Z/p\Z$ for a prime $p$, or a subring of the rational numbers, $\Q$.
    Following this decomposition, the arithmetic square, a technique previously developed by Sullivan, Zabrodsky, and others, is used to reintegrate this information, \cite{Sul70}.
    
    In this context, Bousfield and Kan also introduced the concepts of a solid ring and the core of a ring $R$, the maximal solid subring of $R$, noted $cR$, \cite{BK71}.
    They proved that for every commutative ring $R$ the $R$-completion is equivalent to the $cR$-completion.
    Furthermore, they established that any solid ring is essentially a combination of copies of finite rings $\Z/p\Z$ and subrings of $\mathbb{Q}$.
    This finding underscores that these are the crucial cases when working with the $R$-completion.
    This paper will provide a more precise and detailed classification of solid rings.\\

    Since its inception, the concept of solidity has permeated various areas of mathematics.
    \begin{itemize}
        \item In algebraic geometry, solid rings correspond to subterminal affine schemes, and their classification helps in turn to classify subterminal objects in the broader category of schemes, a task we undertake in subsequent work, \cite{BG25S}.

        \item A generalization of solidity through the concept of epimorphisms is the work of John R. Isbell is his series of papers titled Epimorphisms and Dominions.
        Storrer proved that, in the broader category of not necessarily commutative rings, solid rings are necessarily commutative, \cite[Proposition 1.3]{Sto73}. 
        
        \item The solid ring $c(\Q \times \prod\limits_{p \in \P} \Z/p\Z)$ is a \emph{meadow} in the sense of \cite{BHT09} a generalization of the concept of a field consisting of a ring with all weak inverses, this is a concept for number systems with division using only equations, allowing for $0^{-1} = 0$.
        It is in fact the \emph{meadowization} of $\Z$, \cite{5046853}.  

        \item In \cite{Gut15} Gutiérrez defined solid monoids, showing their bijection with smashing localizations and mapping colocalizations, and obtained more precise results for solid ring spectra in the stable category.

        \item 
        Indeed, many properties of symmetric monoidal categories can be described through the actions of solid semirings, see \cite{Ber18}.

        \item
        Solid rings also emerge in the context of Hopf monads, \cite{HL23}.\\
    \end{itemize}
    In their original work, Bousfield and Kan classified solid rings into four types.
    The first three families are (1) cyclic rings, (2) subrings of the rationals, and (3) product rings $\Z[J^{-1}] \times \Z/n\Z$, where $J$ is any set of primes that contains all prime factors of $n \geq 2$.
    The fourth type was given by any of the following two equivalent definitions:
    
    \begin{enumerate}
        \item[(4a)] Any colimit given by a diagram of rings of type (1), (2), or (3) that does not fall into any of the three types defined previously.

        \item[(4b)] The \emph{core} of the infinite product rings $\Z[J^{-1}] \times \prod\limits_{p \in K} \Z/p^{e_p}\Z$, where $K \subseteq J$ are any two nested infinite sets of primes and the exponents $\{e_p\}_{p \in K}$ are positive integers.
    \end{enumerate}

    While the first three types are straightforward, this fourth class of solid rings has remained comparatively obscure.
    The purpose of this paper is to revisit this classification and provide a more explicit and accessible characterization of this fourth type.\\



    
    The key insights on this paper are Proposition \ref{prp:co_solid_2}, which roughly says that in a given diagram of solid rings the arrows between them do not bring any new information, and Theorem \ref{thm:submain}, which characterizes the core of a given ring by the \emph{basic} solid rings that are contained in our ring:
    
    \begin{definition}
        A solid ring is called \emph{basic} if it is either $\Z/n\Z$ for a positive integer $n$ or $\Z[1/p] \times \Z/p^a\Z$ for a prime $p$ and a nonnegative integer $a$.
    \end{definition}
    
    From this perspective, we show that solid rings naturally decompose into two categories:
    \begin{itemize}
        \item On the one side, we have the finite \emph{basic} solid rings, given by $\Z/n\Z$ for a positive integer $n$, which account for all solid rings of type (1).
    
        \item On the other side, we have arbitrary coproducts of infinite \emph{basic} solid rings, $\Z[1/p] \times \Z/p^a\Z$ for a prime $p$ and a nonnegative integer $a$.
        In such coproducts, every prime appears at most once.
        These rings encompass all solid rings of type (2), (3), or (4), depending on whether the number of positive exponents is zero, finite, or infinite, respectively.
    \end{itemize}

    Moreover, such a coproduct admits a natural interpretation as the colimit of a tower of inclusions, offering a concrete and structured description of solid rings of type (4).\\

    The coproduct interpretation provides a systematic way to compute the core of any ring as the coproduct of all the basic solid rings it contains.
    For a ring $R$ with $c:=\chara(R) \not= 0$, the core is the finite basic solid ring $\Z/c\Z$.
    In contrast, for a ring of characteristic zero, the core is determined by the infinite basic solid rings contained within it.

    We apply this framework to compute the core of several rings, including an explicit determination of the core of an infinite product ring of the form described in (4b), viewed as a subring of the original ring.\\

    Let $\P$ be the set of prime numbers.
    In Subsection \ref{ss:inf_prod} we are able to explicitly compute the core of the ring
    $$
    \Q\times \prod\limits_{p \in \P} \Z/p\Z,
    $$
    which consists precisely of those elements $(r/s, \{a_p\}_{p \in \P})$ such that only finitely many of the ${a_p}'s$ differ from $r/s$ modulo $\Z/p^{e_{p}}\Z$.
    Note that this condition is well-defined, as the quotient is undefined in only finitely many factors; specifically, those corresponding to the prime divisors of $s$, which are finite in number.
    This answers a question posed by John Baez in the comments section of \cite{Bae21}.

    \subsection{Acknowledgments}
    The author thanks Professor Ramón Flores for many helpful discussions and advice in the writing of this paper.
	
	\section{Solid rings}
	\noindent

    \begin{definition}
		A ring\footnote{Our rings are always assumed to be commutative, associative and unital.} $R$ is called \emph{solid} if for any other ring $S$ there exists at most one morphism $R \to S$.
    \end{definition}

    \begin{definition}
        Let $R$ be a ring.
        \begin{enumerate}
            \item An element $x \in R$ is called \emph{absolute} if for any other ring $S$ the image of $x$ under any morphism $R \to S$ is fixed.
            
            \item The \emph{core} of $R$, denoted $cR$, is the subring of absolute elements.
        \end{enumerate}
	\end{definition}

    \begin{lemma}
        A ring is solid if and only if all of its elements are absolute.
    \end{lemma}
    \begin{proof}
        $\implies$
        If $R$ is solid then for any other ring $S$ there exists at most one morphism $R \to S$ and thus the image of any element $x \in R$ under any morphism $R \to S$ is fixed.

        $\impliedby$
        If all the elements of $R$ are absolute then for any other ring $S$ the image of any element $x \in R$ is uniquely defined and thus there exists at most one morphism $R \to S$.
    \end{proof}

    We start by giving an equivalent characterization of absolute elements by making use of the coproduct. Recall that in the category of rings, a coproduct exists and that it coincides with the tensor product over the initial ring $\Z$.

    \begin{lemma}
        Let $R$ be a ring.
        An element $x \in R$ is absolute if and only if $x \otimes 1 = 1 \otimes x$ in $R \otimes_\Z R$.
    \end{lemma}
    \begin{proof}
        $\implies$ If an element $x \in R$ is absolute then considering the two projection morphism $R \to R \otimes_\Z R$ to the left and right coordinate, respectively, we obtain $x \otimes 1 = 1 \otimes x$ in $R \otimes_\Z R$.
        
        $\impliedby$ For any ring $S$ and any two morphism $f,g: R \to S$ we have a commutative diagram through the coproduct
    \begin{center}
        \begin{tikzcd}
                                                                 & R \arrow[d, "1 \otimes -"'] \arrow[rdd, "g", bend left] &   \\
        R \arrow[r, "- \otimes 1"] \arrow[rrd, "f"', bend right] & R \otimes_\Z R \arrow[rd, dashed]                       &   \\
                                                                 &                                                         & S
        \end{tikzcd}
    \end{center}
    and thus for $x \in R$ we have that from $x \otimes 1 = 1 \otimes x$ in $R \otimes_\Z R$ follows that $x$ is absolute.
    \end{proof}

    This characterization yields a description of the core of a ring $R$ more suited for computational purposes,
    $$
    cR = \{x \in R : x \otimes 1 = 1 \otimes x \textup{ in } R \otimes_\Z R\},
    $$
    nevertheless, recall that equality in the tensor product is, in general, not computable, even when the modules involved are computable, see \cite{Bac94}.\\

    Recall the categorical theoretic definition of an epimorphism as a morphism $f : X \to Y$ such that for all objects $Z$ and all morphisms $g_1, g_2: Y \to Z$,
    $$
    g_1f = g_2f \implies g_1 = g_2.
    $$
    Although in the category of sets this concept corresponds exactly to surjective functions, in the category of rings the inclusion $\Z \to \Q$ is a ring epimorphism which is not surjective.

    \emph{Solidity} can also be defined in other categories, an object is solid in a category with an initial object if the unique morphism from the initial object to our object is an epimorphism, in our case we care when the unique morphism $\Z \to R$ is an epimorphism.\\
    
    The following two lemmas are categorical in nature and apply more generally beyond the category of rings.
    These lemmas establish fundamental properties of colimits when the underlying objects are solid.

    \begin{lemma}
    \label{lmm:co_solid_1}
        Any colimit of solid rings is solid.
    \end{lemma}
    \begin{proof}
        Given a colimit diagram of solid rings, any morphism from the colimit to another ring $S$ is uniquely determined by the universal property of colimits. Since each object in the diagram admits at most one morphism to $S$, their uniqueness is preserved in the colimit, ensuring that it remains solid.
    \end{proof}

    The next proposition is central in our work, it states that in a diagram of solid rings we can ignore or rearrange arrows between objects without affecting the colimit.
    This follows from the fact that, due to solidity, morphisms between the objects of our diagram do not add new structural information:
    The colimit of a diagram captures the data of morphisms from each object in the diagram and,
    since each solid ring has at most one morphism to any other ring, the colimit structure is determined purely by the objects themselves rather than by the arrows between them.
    Consequently, we can discard these arrows and work directly with the coproduct of the underlying objects.
    Moreover, if a solid ring appears more than once in a coproduct diagram, the only possible morphism between two of such copies is the identity.
    Hence, we only need to consider one copy in the coproduct.

    \begin{proposition}
    \label{prp:co_solid_2}
        The colimit of a diagram of solid rings is isomorphic to the coproduct of the underlying rings.
        Furthermore, the coproduct of two copies of a solid ring is the solid ring itself.
    \end{proposition}
    \begin{proof}
        Let $D$ be a diagram of solid rings and consider $D'$ to be the same diagram without the arrows.
        Write $c$ for the object in the colimit of $D$ and $c'$ for the object in the colimit of $D'$, i.e. $c'$ is the coproduct of the underlying objects of $D$, and for any object $d$ in $D$ (or $D'$ indistinctively), let us write $c_d$ and $c'_d$ for the morphisms from $d$ to $c$ and $c'$ given by the colimit constructions, respectively.
        We always have an arrow $a: c' \to c$ such that for any object $d$ of $D'$ we have $ac'_d = c_d$ and solidity gives us an arrow $b: c \to c'$, since for any morphism $f: d \to d'$ in $D$ we have that $c'_{d'} \circ f = c'_d$ because $d$ is solid, such that for any object $d$ of $D$ we have $bc_d = c'_d$.
        Since for any object $d$ we have that $bac'_d = bc_d = c'_d$ then $ba = 1_{c'}$ and in the same way since $abc_d = ac'_d = c_d$ we have that $ab = 1_c$.
        And we are done.
    \end{proof}
    %

    \subsection{Coproducts of solid rings of types (1), (2), and (3)}
    \label{sbs:cop}

    We now recall the characterization of solid rings of types (1) and (2), and (3) from the introduction and compute their coproducts explicitly.
    In each case, we will see that a finite coproduct of rings of a given type remains within the same type.
    The proofs of the following lemmas are straightforward calculations of finite coproducts.

    \begin{lemma}
    \label{lmm:co_1}
        The coproduct of any set $\{\Z/n_i\Z\}_{i \in I}$ of solid rings of type (1) can be written as $\Z/g\Z$, where $g = \gcd\limits_{i \in I} n_i$.
    \end{lemma}

    \begin{lemma}
    \label{lmm:co_2}
        The coproduct of any set $\{\Z[J_i^{-1}]\}_{i \in I}$ of solid rings of type (2) can be written as $\Z[J^{-1}]$, where $J = \bigcup\limits_{i \in I} J_i$.
        In particular, $\Z[J^{-1}]$ can be thought as the coproduct of the rings $\{\Z[1/p]\}_{p \in J}$.
    \end{lemma}

    For the last lemma we use the standard notation $v_p(n)$ for a given positive integer $n$ at a prime number $p$ to mean the exponent of $p$ in the prime factorization of $n$.

    \begin{lemma}
    \label{lmm:co_3}
    Any solid ring of type (3), $\Z[J^{-1}] \times \Z/n\Z$ where $J$ is any set of primes that contains all prime factors of $n$, can be written as the coproduct of $\Z[(J\setminus K)^{-1}]$ and $\{\Z[1/p] \times \Z/p^{e_p}\Z\}_{p \in K}$, where $K$ is the finite set of primes that divides $n$ and $e_p := v_p(n)$.
    \end{lemma}
    \begin{proof}
        Let $A$ and $B$ be disjoint sets of primes and let $m$ and $n$ be positive integers such that all prime factors of $m$ are in $A$ and all prime factors of $n$ are in $B$.
        Then
        $$
        (\Z[A^{-1}] \times \Z/m\Z) \otimes (\Z[B^{-1}] \times \Z/n\Z)
        $$
        is equal to
        $$
        (\Z[A^{-1}] \otimes \Z[B^{-1}])
        \times (\Z[A^{-1}] \otimes \Z/n\Z)
        \times (\Z/m\Z \otimes \Z[B^{-1}])
        \times (\Z/m\Z \otimes \Z/n\Z).
        $$
        Since every $p \in A$ is already invertible in $\Z/n\Z$, every $q \in B$ is also invertible in $\Z/m\Z$, and $\gcd(m,n) = 1$, this is equal to
        $$
        \Z[(A \cup B)^{-1}]
        \times \Z/n\Z
        \times \Z/m\Z
        \times 0,
        $$
        and applying $\gcd(m,n) = 1$ once again we get
        $$
        \Z[(A \cup B)^{-1}]
        \times \Z/mn\Z.
        $$
        Applying this fact repeatedly to the rings in our statement we prove the lemma.
    \end{proof}

    These lemmas, together with the fact that colimits commute with colimits, allow us to further simplify any colimit diagram of solid rings of types (1), (2), and (3).
    Specifically, such a colimit can be written as the colimit of a diagram containing at most one solid ring of type (1), $\Z/n\Z$, at most one solid ring of type (2), $\Z[J\setminus K]$ for any two nested sets of primes $K \subseteq J$, and all remaining rings are of the form $\{\Z[1/p] \times \Z/p^{e_p}\Z\}_{p \in K}$ where the exponents are positive integers.

    At this stage, one may observe that solid rings of types (2) and (3) can be unified into a broader framework.
    Both types can be viewed as coproducts of rings of the form
    $$
    \{\Z[1/p] \times \Z/p^{e_p}\Z\}_{p \in J},
    $$
    where only finitely many exponents are positive integers and the rest are zero.
    When all exponents $e_p$ are zero, we recover rings of type (2), while allowing a finite non-zero amount of them to be positive corresponds to rings of type (3).

    This insight suggests a natural way to decompose solid rings into more fundamental building blocks.
    In the following subsection, we dive deeper in the basic solid rings, already defined in the introduction, which serve as the essential components from which all solid rings can be constructed.
    These will provide a conceptual foundation for understanding the structure of solid rings in full generality.\\

    We conclude this section by stating a key structural result, already present in the literature, which ensures that every solid ring can be built from these fundamental cases.
    For more detail in its proof we refer to the original work.
    
    \begin{proposition}[{\cite[Proposition 3.1]{BK71}}]
    \label{prp:co_solid_3}
    \label{prp:co_solid_4}
        Every solid ring can be written as a colimit of solid rings of types (1), (2), and (3) contained in it.
    \end{proposition}
    \begin{proof}
    	Every ring can be written as the colimit of all its finitely generated subrings.
    	Thus, the core of our ring will be the corresponding colimit of the cores of those rings, which, in the original paper, are proved to be of type (1), (2) and (3).
    	Since these cores are subrings of the finitely generated subrings that we considered, they are subrings of our ring, and thus why the "contained in it".
    \end{proof}
   
    \section{Basic solid rings}

    In the previous section, we presented a unified perspective.
    At its foundation lie the \emph{basic} solid rings, which serve as the fundamental building blocks of all solid rings.
    These are divided into two types:

    \begin{itemize}
        \item \emph{Finite basic solid rings}, given by the rings
        $$
        \Z/n\Z
        $$
        for a positive integer $n$, which correspond to solid rings of type (1).

        \item \emph{Infinite basic solid rings}, given by the rings
        $$
        \Z[1/p] \times \Z/p^a\Z
        $$
        for a prime $p$ and a nonnegative integer $a$.
    \end{itemize}

    As proven in the previous section, coproducts of infinite basic solid rings where at most a finite amount of the exponents are non-zero produce all solid rings of types (2) and (3).
    This, together with the fact that every solid ring can be described as a colimit of rings of types (1), (2), and (3), means that every solid ring can ultimately be expressed in terms of these basic building blocks as a colimit.
    Furthermore, these basic solid rings in Lemma \ref{lmm:co_2} and Lemma \ref{lmm:co_3} whose coproduct give solid rings of type (2) and (3), respectively, are also subrings of such coproducts.
    This yields the following refinement of Proposition \ref{prp:co_solid_4}:

    
    \begin{theorem}
    	\label{thm:submain}
    	The core of a ring $R$ is the coproduct of all its basic solid subrings.
    \end{theorem}
    \begin{proof}
    	The lattice of solid subrings of a ring $R$ has a maximal element, its core $cR$, and thus the coproduct of all these rings is precisely $cR$.
    	Furthermore, since colimits commute with colimits, we can apply Proposition \ref{prp:co_solid_4} and substitute every such solid subring by solid subrings of type (1), (2) and (3) that are contained in it, and after another step doing a substitution of each such subring by its corresponding basic solid subrings and deleting repeated copies, Proposition \ref{prp:co_solid_2}, we obtain the assertion.
    \end{proof}

    %
    %


    \subsection{Maps from basic solid rings}
    
    In this subsection, we analyze how the existence of a morphism from a basic solid ring to a given ring $R$ constrains the structure of $R$.  
    By Theorem \ref{thm:submain} above, understanding the morphisms from these building blocks is equivalent to understanding morphisms from arbitrary solid rings.

    \begin{lemma}
    \label{lm:explicit_meaning_fin}
        There exists a morphism $\Z/m\Z \to R$ if and only if the characteristic of $R$ divides $m$.
    \end{lemma}
    \begin{proof}
        There is a morphism $\Z/m\Z \to R$ if and only if $m = 0$ in $R$, the statement follows.
    \end{proof}

    \begin{lemma}
    \label{lm:explicit_meaning}
        There exists a morphism $\varphi: \Z[1/p] \times \Z/p^a\Z \to R$ if and only if there exists an element $r_{p,a} \in R$ satisfying
        $$
        pr_{p,a}^2 = r_{p,a} \textup{ and } p^{a+1}r_{p,a} = p^a.
        $$
    \end{lemma}
    \begin{proof}
        $\implies$
        Let $r_{p,a} := \varphi(\frac{1}{p},0)$.
        
        $\impliedby$
        For the other direction, consider the morphism given by
        $$
        \varphi:(m/p^t, n) \mapsto pr_{p,a}^{t+1}m + (1-pr_{p,a})n.
        $$ 
        Repeatedly using the condition $pr_{p,a}^2 = r_{p,a}$ yields that given two different ways of writing the fractional part $m/p^t$ and $mp^k/p^{t+k}$ give the same value, i.e.
        $$
        \varphi(m/p^t, n) = pr_{p,a}^{t+1}m + (1-pr_{p,a})n = pr_{p,a}^{t+k+1}mp^k + (1-pr_{p,a})n = \varphi(mp^k/p^{t+k+1},n).
        $$
        and thus we have well-definiteness.
        We leave the straightforward check that it indeed satisfies the conditions for it to be a ring morphism to the reader.\\
    \end{proof}
    
    These two lemmas help us determine when a morphism exists between any two basic solid rings.
    Let us investigate the four possible simple cases:\\

    \begin{enumerate}
        \item $\Z[1/p] \times \Z/p^{a}\Z \to \Z/[1/q]$.
        Take $x \in \Z[1/q]$ to be $r_{p,a}$ as in the notation of Lemma \ref{lm:explicit_meaning}.
        The first equation yields $px^2 = x$ in $\Z/[1/q]$, which can only happen whenever $q = p$ and $x = 1/p$.\\

        \item $\Z[1/p] \times \Z/p^{a}\Z \to \Z/n\Z$.
        Take $x \in \Z[1/q]$ to be $r_{p,a}$ as in the notation of Lemma \ref{lm:explicit_meaning}.
        
        If $a < v_p(n)$, then $\frac{n}{p^{a+1}}$ is an element of $\Z/n\Z$ and the second equation yields
        $$
        0 = nx = \frac{n}{p^{a+1}} p^{a+1} x = \frac{n}{p^{a+1}} p^a = \frac{n}{p},
        $$
        which is not true in $\Z/n\Z$ and so no such morphism exists.

        If $a \geq v_p(n)$ then by taking the factorization $\Z/n\Z \cong \Z/p^{v_p(n)}\Z \times \Z/\frac{n}{p^{v_p(n)}}\Z$ we get a morphism by the universal property of the product induced by the composition with the natural morphisms $\Z[1/p] \to \Z/\frac{n}{p^{v_p(n)}}\Z$ and $\Z/p^a\Z \to \Z/p^{v_p(n)}\Z$ after the respective projections.
        
        \item $\Z/m\Z \to \Z[1/p]$. By Lemma \ref{lm:explicit_meaning_fin} there is no such morphism.

        \item $\Z/m\Z \to \Z/n\Z$. By Lemma \ref{lm:explicit_meaning_fin} such a morphism exists if and only if $n | m$.\\
    \end{enumerate}

    Since a morphism to a product is equivalent to a pair of morphisms, one to each coordinate, the only cases when there is a morphism between basic solid rings are:

    \begin{itemize}
        \item A morphism
            \[
    \tag{3.A}\label{S1}
    \Z[1/p] \times \Z/p^{a}\Z \twoheadrightarrow \Z[1/p] \times \Z/p^{b}\Z.
    \]
        whenever $b \leq a$.
        
        \item A morphism
    \[
    \tag{3.B}\label{S2}
    \Z[1/p] \times \Z/p^{a}\Z \twoheadrightarrow \Z/n\Z,
    \]
        whenever $a \geq v_p(n)$.

        \item A morphism
    \[
    \tag{3.C}\label{S3}
    \Z/m\Z \twoheadrightarrow \Z/n\Z,
    \]
        whenever $n \mid m$.
    \end{itemize}
    In all these cases, it is easy to check that our morphism is surjective and, furthermore, by the solidity of the domain it is unique.\\

    These observations immediately determine the core of any ring with nonzero characteristic:
    \begin{lemma}
    \label{lem:char_not0_core}
        Let $R$ be a ring with characteristic $c := \chara(R) \not= 0$.
        Then the core of $R$ is given by  
        $$
        cR = \Z/c\Z.
        $$
    \end{lemma}
    \begin{proof}
    	By Theorem \ref{thm:submain} we have to compute the colimit of the basic solid rings it contains.
    	On one hand, $\Z/c\Z$ is the only finite basic solid ring contained in $R$.
    	On the other hand, no infinite basic solid ring is strictly contained in $R$, since the characteristic of these rings is distinct than $0$ and, thus, no injection is possible.
    	We are done.
    \end{proof}

    The rest of the paper is devoted to compute the core of a ring with zero characteristic, a task that we successfully undertake.
    We start by establishing when an infinite basic solid ring can be embedded into $R$:

    \begin{lemma}
    \label{lm:char0_core}
    Let $R$ be a ring of characteristic zero, and assume there exists a morphism
    $$
    \Z[1/p] \times \Z/p^a\Z \to R
    $$
    for some prime $p$ and a nonnegative integer $a$.
    Then there exists a nonnegative integer $b \leq a$ such that there is an injective morphism
    $$
    \Z[1/p] \times \Z/p^b\Z \hookrightarrow R.
    $$
    \end{lemma}
    \begin{proof}
    Any morphism can be factored as a quotient map followed by an inclusion.
    If an element whose first coordinate is nonzero belongs to the kernel, then $\chara(R) \not= 0$, which contradicts our assumption.
    Hence, the kernel must be contained in the $\Z/p^a\Z$ component, ensuring the factorization of this component through a subring $\Z/p^b\Z$.
    \end{proof}

    In the next section, we leverage the tower of inclusions perspective to provide a precise characterization of solid rings of type (4).

    \section{Solid rings of type (4)}

    It remains to analyze coproducts of infinite basic solid rings in which each prime appears at most once and an infinite number of the exponents are positive.
    In this section, we begin by introducing the tower of inclusions perspective, which provides a concrete description of solid rings of type (4).
    We then verify that these rings are indeed genuinely new.
    Finally, we examine how these rings naturally embed as subrings within the original rings proposed by Bousfield and Kan in their foundational work.
    \subsection{The tower of inclusions perspective}

    Consider a family of infinite basic solid rings
    $$
    \{\Z[1/p] \times \Z/p^{e_p}\Z\}_{p \in J},
    $$
    where $J$ is an infinite collection of primes, and let $K := \{p \in J : e_p \geq 1\}$ be an infinite subset of $J$.  
    Without loss of generality, assume $K = \{p_1, p_2, p_3, \dots\}$ is ordered by increasing magnitude, so that we may define $K_n := \{p_i\}_{i = 1}^n$ and $J_n := (J \setminus K) \cup K_n$.
    
    Because colimits commute with colimits, we can first take the coproduct of all rings corresponding to primes in $J_0 = J \setminus K$ as a single step, obtaining $\Z[J_0^{-1}]$.
    Then, by sequentially taking coproducts with one additional prime at each step, we obtain the following directed system of inclusions:
    \[
    \tag{3.A}\label{SD}
    \begin{tikzcd}
    {\Z[J_0^{-1}]} \arrow[r, hook] & \cdots \arrow[r, hook] & {\Z[J_n^{-1}] \times \prod\limits_{p \in K_n} \Z/p^{e_p}\Z} \arrow[r, hook] & \cdots,
    \end{tikzcd}
    \]
    where each inclusion is induced by the universal property of coproducts.
    %
    %
    
    Let $S$ be the colimit of this sequential system, its elements can be understood as elements of some $\Z[J_n^{-1}] \times \prod\limits_{p \in K_n} \Z/p^{e_p}\Z$ in their minimal form, i.e. those that do not belong to the image of any earlier morphism.  
    The ring operations in $S$ are inherited from the colimit structure, meaning that addition and multiplication between any two elements take place in the largest ring containing both.

    \subsection{Classification of solid rings of type (4)}

    We refer to the data given by these infinite basic solid rings as $(J, \{e_p\}_{p \in J})$ where $J$ is an infinite set of primes and $\{e_p\}_{p \in J}$ is a family of nonnegative exponents, with the condition that infinitely many of them are positive.
    
    To confirm that these data define genuinely new rings, distinct from those of type (1), (2), or (3), as well as from any other type (4) ring described by different data, we examine their torsion subring and the corresponding quotient, \cite[Example 3.2]{BK71}:
    \begin{itemize}
        \item The torsion subring:  
        $$
        S^t = \bigoplus\limits_{p \in K} \Z/p^{e_p}\Z.
        $$

        \item The quotient by torsion:
        $$
        S/S^t = \Z[J^{-1}].
        $$
    \end{itemize}
    These invariants show that different choices of \( (J, \{e_p\}) \) yield non-isomorphic rings, completing the classification.


    \subsection{The core of the infinite product $\Z[J^{-1}] \times \prod\limits_{p \in K} \Z/p^{e_p}\Z$}
    \label{ss:inf_prod}
    \noindent
    
    Let
    $$
    R := \Z[J^{-1}] \times \prod\limits_{p \in K} \Z/p^{e_p}\Z.
    $$
    In this subsection, we give a precise description of the core $cR$.

    Since $\chara{R} = 0$, no finite basic solid ring embeds into $R$, so we only need to determine which rings of the form 
    $$
    \Z[1/p] \times \Z/p^a\Z
    $$
    for some prime $p$ and $a \geq 0$, are contained in $R$.
    
    For a morphism from such a ring to $R$ to exist, the map must be defined coordinate-wise, which requires $p \in J$ and $e_p \leq a$.
    The injection occurs when $a = e_p$, and the coproduct of these inclusions gives precisely the largest solid subring of $R$, namely $cR$.\\
    
    To determine the explicit inclusion $cR \subseteq R$, consider the sequential diagram of inclusions defined in (\ref{SD}) above and, for any nonnegative integer $n$, the inclusion
    $$
    \Z[J_n^{-1}] \times \prod\limits_{p \in K_n} \Z/p^{e_p}\Z \hookrightarrow R
    $$
    defined coordinate-wise in the only possible way.
    This gives rise to the following factorization:
    $$
    \begin{tikzcd}
    {\Z[J_n^{-1}] \times \prod\limits_{p \in K_n} \Z/p^{e_p}\Z} \arrow[d, hook] \arrow[rd, hook] &   \\
    S \arrow[r, hook]                                                                            & R
    \end{tikzcd},
    $$
    where $S$ is the solid ring defined as the colimit of the diagram given by its corresponding tower of inclusions, meaning the bottom inclusion is precisely the inclusion $cR \subseteq R$.
    
    With this description of $S$ as the colimit of a tower of inclusions we have that elements of $S$ are identified with elements living in one of the rings on that tower, where operations take place in the ring where the element that lives further lives.
    
    Then any element of $S$ can be seen inside $R$ as the image of an element 
    $$
    (r/s, a_1, \dots, a_n) \in
    \Z[J_n^{-1}] \times \prod\limits_{p \in K_n} \Z/p^{e_p}\Z.
    $$
    Looking at this element inside $R$ through the factorization map $S \hookrightarrow R$ we obtain an element $(r/s, a_1, a_2, a_3, \dots) \in R$ such that $a_m = r/s$ in $\Z/p_i^{e_{p_i}}\Z$.
    We get that $cR$ consists precisely of all the elements $(r/s, a_1, a_2, a_3, \dots) \in R$ such that only finitely many of the $a_i's$ are distinct to $r/s$ modulo $\Z/p_i^{e_{p_i}}\Z$, which is well defined in all but finitely many of the factors since $s$ has a finite amount of prime factors.
    
    This reasoning solves a question posed by John Baez in the comments section of \cite{Bae21}.
    More concrete examples and computations can be found at the end in Section \ref{sec:compu}.

    \section{All solid rings}

    As noted in Lemma \ref{lem:char_not0_core}, the core of a ring with nonzero characteristic is fully determined by its characteristic.
    On the other hand, as established in the previous section, the core of a characteristic zero ring $R$ is precisely determined by the set of rings of the form
    $$
    \Z[1/p] \times \Z/p^a\Z
    $$
    for some prime $p$ and $a \geq 0$, that embed into $R$.
    More generally, the core of any ring $R$ is completely determined by the basic solid rings that embed into it.
    Indeed:
    \begin{itemize}
        \item
        For a ring with nonzero characteristic, no infinite basic solid ring embeds into it.

        \item 
        For a ring of characteristic zero, no finite basic solid ring embeds into it.
    \end{itemize}

    Nevertheless, for a ring with nonzero characteristic, it is still meaningful to consider morphisms from infinite basic solid rings into it, even though these morphisms are not embeddings, as the following lemma shows:
    
    \begin{lemma}
        Let $R$ be a ring with characteristic $c := \chara{R} \not= 0$.
        Then for any prime $p$, there exists a morphism
        $$
        \Z[1/p] \times \Z/p^{a}\Z \to R
        $$
        if and only if $a \geq v_p(c)$.
    \end{lemma}
    \begin{proof}
        Since we have a tower of surjections
        $$
        \cdots \twoheadrightarrow
        \Z[1/p] \times \Z/p^{n+1}\Z \twoheadrightarrow
        \Z[1/p] \times \Z/p^{n}\Z \twoheadrightarrow
        \cdots \twoheadrightarrow \Z[1/p]
        $$
        it suffices to verify that:
        \begin{enumerate}
            \item[1.]
            A morphism exists for $a = v_p(c)$.

            \item[2.]
            No morphism exists for $a = v_p(c) - 1$ whenever $p \mid c$.
        \end{enumerate}
        The existence of a morphism for $a = v_p(c)$ follows directly from the composition
        $$
        \Z[1/p] \times \Z/p^{v_p(c)}\Z \twoheadrightarrow
        \Z/c\Z \hookrightarrow
        R.
        $$
        To show that no morphism exists for $a = v_p(c) - 1$ when $p \mid c$, suppose otherwise.
        Then, by Lemma \ref{lm:explicit_meaning}, such a morphism would imply the existence of an element $r \in R$ such that
        $$
        p^{v_p(c)}r = p^{v_p(c)-1}.
        $$
        Writing $c = p^{v_p(c)}d$ with $\gcd(p,d) = 1$, we obtain
        $$
        0 = cr = dp^{v_p(c)}r = dp^{v_p(c)-1} \not= 0,
        $$
        a contradiction.
    \end{proof}

    This information leads to a new classification of solid rings, which we now summarize in terms of an exponent function and a binary parameter:

    \subsection{A classification of solid rings}
    \label{ssc:class}

    The following way of presenting the data will be useful in the following section.
    A solid ring is given by the following data:
    \begin{enumerate}
        \item
        An exponent function $e: \P \to \{0, 1, 2, \dots\} \cup \{+\infty\}$.
        \item 
        A binary parameter $q \in \{0,1\}$, where $q = 0$ if either:
        \begin{itemize}
            \item 
            There exists a prime $p$ with $e(p) = +\infty$, or

            \item 
            There are infinitely many primes $p$ with $e(p) > 0$.
        \end{itemize}
    \end{enumerate}

    These data uniquely specify a solid ring, where the parameter $q$ distinguishes between finite and infinite constructions:
    
    \begin{itemize}
        \item
        If $q = 0$, the ring is obtained as a coproduct of the infinite basic solid rings
        $$
        \Z[1/p] \times \Z/p^{e(p)}\Z,
        $$
        where $p$ runs through all the prime numbers such that $e(p) < +\infty$.

        \item 
        If $q = 1$, the solid ring is finite and the exponent function provides the prime factorization data, the ring is the cyclic ring obtained as the product of the rings
        $$
        \Z/p^{e(p)}\Z,
        $$
        where $p$ runs through the finite amount of prime numbers with $e(p) > 0$.
    \end{itemize}

    \section{Limits and colimits of solid rings}

    As described in Subsection \ref{ssc:class} above, we can represent the core of a ring $R$ as a pair $(cR, q)$ where:
    \begin{itemize}
        \item 
        The function
    $$
    cR: \P \to \{0, 1, 2, \dots\} \cup \{+\infty\},
    $$
    assigns to each prime $p$ the minimal exponent $a_p$ such that there exists a morphism
    $$
    \Z[1/p] \times \Z/p^{a_p}\Z \to R.
    $$
    
        \item 
    The binary parameter $q$ indicates whether the characteristic of $R$ is zero or nonzero.
    \end{itemize}
	
	In this section, we will denote with $cR$ not the core of $R$ but its corresponding function.

    \subsection{Limits}

    The following lemma describes how the function associated with a limit relates to the functions of the underlying objects in the diagram.
    Notably, in this first result, the parameter $q$ does not play a role.

    \begin{lemma}
    \label{lm:limit_c}
        Let $\{R_i\}_{i \in I}$ be a family of rings that forms a diagram whose limit is $R$.
        Then
        $$
        cR = \sup\limits_{i \in I} cR_i.
        $$
    \end{lemma}
    \begin{proof}
        Fix a prime $p$.
        We prove that
        $$
        cR(p) = \sup\limits_{i \in I} cR_i(p).
        $$

        Since $R$ is the limit of the $R_i$, there are natural projections $R \to R_i$, and for any $i \in I$, we have
        $$
        cR(p) \geq cR_i(p),
        $$
        because a morphism $\Z[1/p] \times \Z/p^{a_p}\Z \to R$ induces a corresponding morphism into each $R_i$.
        This gives one side of the inequality.

        Let $M :=\sup\limits_{i \in I} cR_i(p)$.
        If  $M = +\infty$, the equality follows immediately.
        Otherwise, since $\Z[1/p] \times \Z/p^M\Z$ maps to every $R_i$ in a compatible way with the diagram we are working with since $\Z[1/p] \times \Z/p^M\Z$ is a solid ring, it must factor through the limit $R$, giving a morphism 
        $$
        \Z[1/p] \times \Z/p^M\Z \to R.
        $$
        This shows the other side of the equality,
        $$
        cR(p) \leq \sup\limits_{i \in I} cR_i(p),
        $$
        completing the proof.
    \end{proof}

    We now turn our attention to the parameter $q$, which distinguishes between the cases of characteristic zero and positive characteristic.
    
    \begin{lemma}
    \label{lm:limit_q}
        Let $\{R_i\}_{i \in I}$ be a family of rings forming a diagram whose limit is $R$.
        Then $q = 1$ if and only if every ring $R_i$ in the diagram satisfies $q_i = 1$, $cR(p) < +\infty$ for any prime $p$, and there are at most finitely many primes with $cR(p) > 0$.
    \end{lemma}
    \begin{proof}
        Suppose $q = 1$.
        Then there exists a morphism $\Z/\chara(R)\Z \to R \to R_i$ for each $i \in I$, which implies that $R_i$ has nonzero characteristic.
        Thus $q_i = 1$ for every $i \in I$.
        Furthermore, the conditions $cR(p) < +\infty$ for all primes $p$, and the finiteness of $\{p \in \P : cR(p) > 0\}$ follows directly, since otherwise, we would have $q = 0$, contradicting our assumption.
        
        For the converse, assume that every ring in the diagram satisfies $q_i = 1$.  
        Since we are given that $cR(p) < +\infty$ for all $p$ and that only finitely many primes satisfy $cR(p) > 0$, we can construct the solid ring  
        $$
        \prod\limits_{p \in \P} \Z/p^{cR(p)}\Z,
        $$
        which maps in a compatible way into every $R_i$, and thus it factor through $R$ by the universal property of limits.  
        This ensures that $q = 1$, completing the proof.
    \end{proof}
    
    Both lemmas together provide a complete description of how the exponent function and the parameter $q$ behave under limits, allowing us to determine the core of a limit of rings explicitly in terms of the cores of the underlying objects in the diagram.

    \subsection{Colimits}

    In this subsection we focus on the core construction for arbitrary colimits.
    
    \begin{lemma}
    \label{lm:colimit_c}
        Let $\{R_i\}_{i \in I}$ be a family of rings forming a diagram whose colimit is $R$.
        Then
        $$
        cR \leq \min\limits_{i \in I} cR_i.
        $$
    \end{lemma}
    \begin{proof}
        Let $m_p := \min\limits_{i \in I} cR_i(p)$.
        Then there exists $i \in I$ such that we have a morphism $\Z[1/p] \times \Z/p^{m_p}\Z \to R_i$. Composing with the canonical map $R_i \to R$ induced by the colimit gives a morphism into $R$, and hence the inequality follows.
    \end{proof}

    Notably, for arbitrary colimits, this inequality can not be made into an equality.
    Consider the following pushout diagram, where the arrows are the projections onto the even and odd-indexed components, respectively:
    $$
    \begin{tikzcd}
    \Z^\N \arrow[d, two heads] \arrow[r, two heads] & \Z^{2\N+1} \\
    \Z^{2\N}                                        &           
    \end{tikzcd}
    $$
    The colimit of this diagram is $\Z^{2\N} \otimes_{\Z^\N} \Z^{2\N+1} = 0$, because the image of the element $(0,1,0,1,\dots)$ forces any ring $A$ making the square commute to satisfy $0 = 1$, and hence $A = 0$.\\

    \begin{lemma}
        Let $\{R_i\}_{i \in I}$ be a family of rings forming a diagram whose colimit is $R$.
        Then
        $$
        \max\limits_{i \in I} q_i \leq q.
        $$
    \end{lemma}
    \begin{proof}
        If any $R_i$ has characteristic not equal to zero (i.e., $q_i = 1$), then a cyclic ring maps into $R$ after composition with the natural map induced by the colimit, implying $q = 1$.
    \end{proof}

    Again, for arbitrary colimits, this inequality can not be made into an equality.
    The counterexample is the same as above.
    It shows a diagram of rings of characteristic zero whose colimit is the characteristic one ring (i.e., the zero ring).\\
    
    A couple of questions are left open after this section:
    \begin{enumerate}
        \item Can an arbitrary coproduct of characteristic zero rings have positive characteristic? And in the affirmative case, what about the coproduct of two rings?
        \item Can we obtain an strict inequality in Lemma \ref{lm:colimit_c} when the colimit is an arbitrary coproduct? And in the affirmative case, what about the coproduct of two rings?
    \end{enumerate}


    \section{Computations}
    \label{sec:compu}
    We conclude our work by computing the cores of some well-known characteristic zero rings.
    Recall that the case of rings with nonzero characteristic has already been settled in Lemma \ref{lem:char_not0_core}.

    \subsection{A characteristic zero field}
    Let $K$ be a field of characteristic zero.  
    Since $\Q \hookrightarrow K$, we have $q_K = 0$ and $cK \equiv 0$.  
    Therefore, the core of $K$ is:
    $$
    cK = \Q.
    $$

    \subsection{The $p$-adic integers $\Z_p$}
    
    Recall that $\Z_p$ is the inverse limit of the system:
    $$
    \cdots \twoheadrightarrow \Z/p^{n+1}\Z \twoheadrightarrow \Z/p^n\Z \twoheadrightarrow
    \cdots \twoheadrightarrow \Z/p^2\Z \twoheadrightarrow \Z/p\Z.
    $$
    For each $n$, the core of $\Z/p^n\Z$ is given by:
    $$
    c(\Z/p^n\Z)(q) =
    \begin{cases}
    0, & \text{if $q \not= p$}\\
    n, & \text{if $q = p$}
    \end{cases}.
    $$
    Applying Lemma~\ref{lm:limit_c}, we obtain:
    $$
    c\Z_p(q) =
    \begin{cases}
    0, & \text{if $q \not= p$}\\
    +\infty, & \text{if $q = p$}
    \end{cases}.
    $$
    Hence, $c\Z_p = \Z_{(p)}$, the localization of $\Z$ at the prime ideal $(p)$.
    
    \subsection{The profinite completion of the integers $\widehat{\Z}$}
    Recall that $\widehat{\Z}$ is the product over all primes:
    $$
    \widehat{\Z} := \prod_{p \in \P} \Z_p.
    $$
    A direct application of Lemma~\ref{lm:limit_c} yields:
    $$
    c\widehat{\Z}(q) = 
    +\infty, \text{ for all } q \in \P,
    $$
    which means:
    $$
    c\widehat{\Z} = \Z.
    $$

    \subsection{A polynomial ring over a characteristic zero ring}
    Let $R$ be a ring of characteristic zero, and let $R[\{X_i\}_{i \in I}]$ be a polynomial ring over $R$ in any collection of indeterminates.

    Suppose there exists a morphism from $\Z[1/p] \times \Z/p^a\Z$ to $R[\{X_i\}_{i \in I}]$. By Lemma~\ref{lm:explicit_meaning}, this would imply the existence of a polynomial $F$ satisfying $pF^2 = F$. Since $R$ has characteristic zero, such an identity forces $F$ to be a degree zero polynomial. Thus, the image of any basic solid ring must lie entirely in the subring $R \subseteq R[\{X_i\}]$.
    
    We conclude that
    $$
    c(R[\{X_i\}_{i \in I}]) = cR.
    $$

    
    \printbibliography
\end{document}